\documentclass[a4paper,reqno,11pt]{amsart}
\usepackage{amssymb,amsfonts,verbatim,mathtools}
\usepackage[alwaysadjust]{paralist}
\usepackage[alwaysadjust]{paralist}
\usepackage[hypcap]{caption}
\numberwithin{equation}{section}
\newtheorem{conj}[equation]{Conjecture}
\newtheorem{claim}[equation]{Claim} 

\newtheorem{lem}[equation]{Lemma}
\newtheorem{prop}[equation]{Proposition}
\newtheorem{thm}[equation]{Theorem}
\theoremstyle{definition}
\newtheorem{que}[equation]{Question}
\newtheorem{defn}[equation]{Definition}

\newtheorem{rem}[equation]{Remark}

\newcommand{\cref}[1]{Corollary~\ref{#1}}

\newcommand{\id}{\operatorname{id}}

\begin{document}
\title{On largeness and multiplicity of the first eigenvalue}
\author{Sugata Mondal}

\address{Max Planck Institute for Mathematics, Vivatsgasse 7, 53111 Bonn.}
\email{sugata.mondal\@@mpim-bonn.mpg.de}
\date{\today}

\begin{abstract}
We apply topological methods to study the smallest non-zero number $\lambda_1$ in the spectrum of the Laplacian on finite area hyperbolic surfaces. For closed hyperbolic surfaces of genus two we show that the set $\{ S \in {\mathcal{M}_2}: {\lambda_1}(S) > \frac{1}{4} \}$ is unbounded and disconnects the moduli space ${\mathcal{M}_2}$.
\end{abstract}

\maketitle

\section*{Introduction}
In this paper we identify hyperbolic surfaces with quotients of the Poincar\'{e} upper halfplane $\mathbb{H}$ by discrete torsion free subgroups of PSL$(2, \mathbb{R})$ called {\it Fuchsian groups}. The {\it Laplacian} on $\mathbb H$ is the differential operator $\Delta$ which associates to a $C^2$- function $f$ the function
\begin{equation}
\Delta f(z) = {y^2}(\frac{{\partial^2}f}{{\partial x}^2} + \frac{{\partial^2}f}{{\partial y}^2}).
\end{equation}
For any Fuchsian group $\Gamma$, the induced differential operator on $S = {\mathbb H/ \Gamma}$, $\Delta = {\Delta_S}$ is called the Laplacian on $S$. It is a non-positive operator whose spectrum spec$(\Delta)$ is contained in a smallest interval $( -\infty, -{\lambda_0}(S)] \subset {\mathbb R^{-}} \cup \{0\}$ with ${\lambda_0}(S) \geq 0$. Points in the discrete spectrum will be referred to as an {\it eigenvalue}. In particular this means $\lambda \geq 0$ is an eigenvalue if there exists a non-zero $C^2$-function $ f \in {L^2}(S)$, called a $\lambda$-{\it eigenfunction}, such that $ \Delta f + \lambda f = 0.$ When $0 < \lambda \leq {1/4}$, $\lambda$ is called a {\it small eigenvalue} and $f$ is called a {\it small eigenfunction}.

We shall restrict ourselves to hyperbolic surfaces with finite area. Any such surface $S$ is homeomorphic to a closed Riemann surface $\overline{S}$ of certain genus $g$ from which some $n$ many points are removed. In that case $S$ is called a finite area hyperbolic surface of type $(g, n)$. Each of these $n$ points is called a {\it puncture} of $S$.

The Laplace spectrum of a closed hyperbolic surface $S$ consists of a discrete set:
\begin{equation}
0= {\lambda_0} < {\lambda_1}(S) \leq ... \leq {\lambda_n}(S) \leq ... \infty
\end{equation}
such that ${\lambda_i}(S) \rightarrow \infty$ as $i \rightarrow \infty$. Each number in the above sequence is repeated according to its multiplicity as eigenvalue. The number ${\lambda_i}(S)$ is called the $i$-th eigenvalue of $S$. It is known that the map ${\lambda_i}: {\mathcal{M}_g} \rightarrow {\mathbb{R}}$ that assigns a surface $S \in {\mathcal{M}_g}$ to its $i$-th eigenvalue ${\lambda_i}(S)$ is continuous and bounded \cite{B3}. Hence \begin{equation}
{\Lambda_i}(g) = {\sup_{S \in {\mathcal{M}_g}}} {\lambda_i}(S) < \infty.
\end{equation}
For non-compact hyperbolic surfaces of finite area the spectrum of the Laplacian is more complicated. It consists of both continuous and discrete components (see \cite{I} for detail). However, the part of the spectrum lying in $[0, \frac{1}{4})$ is discrete. Keeping resemblance to above definition for any hyperbolic surface $S$ let us define ${\lambda_1}(S)$ to be the smallest positive number in spec$(\Delta)$. In particular, if ${\lambda_1} < \frac{1}{4}$ then it is an eigenvalue. The function $\lambda_1$, so defined, is bounded by $\frac{1}{4}$ because $S$ has a continuous spectrum on $[\frac{1}{4}, \infty)$. As before we consider the quantity
\begin{equation}
{\Lambda_1}(g, n) = {\sup_{S \in {\mathcal{M}_{g, n}}}} {\lambda_1}(S).
\end{equation}
In \cite{Se} Atle Selberg proved that for any congruence subgroup $\Gamma$ of SL$(2, \mathbb{Z})$
\begin{equation}\label{selberg}
{\lambda_1}(\mathbb{H}/\Gamma) \geq \frac{3}{16}.
\end{equation}
Recall that a congruence subgroup is a discrete subgroup of SL$(2, \mathbb{Z})$ that contains one of the $\Gamma_n$ where
\begin{equation}
{\Gamma_n} = \{
\left(\begin{array}{cc}
a & b\\
c & d \end{array}\right) \in \textrm{SL}(2, \mathbb{Z}): a \equiv 1 \equiv d ~ \textrm{and} ~ b \equiv 0 \equiv c ~ (\text{mod} ~ n)\}
\end{equation}
is the principal congruence subgroup of level $n$. Moreover he conjectured
\begin{conj}\label{sel}
For any congruence subgroup $\Gamma$, ${\lambda_1}(\mathbb{H}/\Gamma) \geq \frac{1}{4}$.
\end{conj}
M. N. Huxley \cite{Hu} proved this conjecture for $\Gamma_n$ with $n \le 6$. Several attempts have been made to prove it (see \cite[Chapter 11]{I} for details) in the general case. The best known bound is $\frac{975}{4096}$ due to Kim and Sarnak \cite{K-S}. This conjecture motivated, in particular, the question of our interest:
\begin{que}\label{ques}
Given any genus $g \ge 2$ does there exist a closed hyperbolic surface of genus $g$ with $\lambda_1$ at least $\frac{1}{4}$ ?
\end{que}
A slightly weaker question than the above one would be: Is ${\Lambda_1}(g) \ge \frac{1}{4}$ ? This question is studied in \cite{BBD} by P. Buser, M. Burger and J. Dodziuk and in \cite{B-M} by R. Brooks and E. Makover. The ideas in \cite{BBD} and \cite{B-M}, in the light of the bound of Kim and Sarnak in \cite{K-S}, provide the following.
\begin{thm}\label{aprox}
Given any $\epsilon>0$, there exists ${N_\epsilon} \in \mathbb{N}$ such that for any $g \ge {N_\epsilon}$ there exist closed hyperbolic surfaces of genus $g$ with ${\lambda_1} \ge \frac{975}{4096} - \epsilon$.
\end{thm}
The constant $\frac{975}{4096}$ in the above theorem can be replaced by $\frac{1}{4}$ if conjecture \ref{sel} is true. Hence it is tempting to conjecture:
\begin{conj}\label{conj}
For every $g \geq 2$ there exists a closed hyperbolic surface of genus $g$ whose $\lambda_1$ is at least $\frac{1}{4}$.
\end{conj}
\begin{rem}
Observe that even if Selberg's conjecture (conjecture \ref{sel}) is true, theorem \ref{aprox} do not provide a positive answer to conjecture \ref{conj}. However it would answer positively the weaker version of our question i.e. it would imply ${\Lambda_1}(g) \geq \frac{1}{4}$, for large values of $g$.
\end{rem}
The existence of genus two hyperbolic surfaces with $\lambda_1 > \frac{1}{4}$ has been known in the literature for sometime \cite{Je}. It is known that the {\it Bolza} surface has $\lambda_1$ approximately $3.8$ (see \cite{S-U} for more details). We consider the subset $\mathcal{B}_2(\frac{1}{4}) =  \{ S \in {\mathcal{M}_2}: {\lambda_1}(S) > \frac{1}{4} \}$ of the moduli space $\mathcal{M}_2$. From the continuity of $\lambda_1$ it is clear that $\mathcal{B}_2(\frac{1}{4})$ is open. Our first result provides better understanding of this set.
\subsection{Eigenvalue branches}
Recall that the moduli space $\mathcal{M}_g$ is the quotient of $\mathcal{T}_g$ by the {\it Teichm\"{u}ller modular group} $M_g$ (see \cite{B3}). We are shifting from the moduli space to the Teichm\"uller space mainly because we wish to talk about analytic paths which involves coordinates and on $\mathcal{T}_g$ one has the Fenchel-Nielsen coordinates (given a pants decomposition) which is easy to describe.

Let $\gamma:[0, 1] \to {\mathcal{T}_2}$ be an analytic path. Since, in this case, $\lambda_1$ is simple as long as small, the function ${\lambda_1}({S^t})$ (${S^t}=\gamma(t)$) is also analytic (see theorem \ref{b-c}) if ${\lambda_1}({S^t}) \le \frac{1}{4}$ for all $t \in [0, 1]$. For higher genus $\lambda_1$ may not be simple even if small (see \S\ref{mult}). Therefore, for an analytic path $\gamma: [0, 1] \to {\mathcal{T}_g}$, ${\lambda_1}({S^t})$ is continuous but need not be analytic even if ${\lambda_1}({S^t}) \le \frac{1}{4}$ for all $t \in [0, 1]$. However we have the following result from \cite[Theorem 14.9.3]{B3}:
\begin{thm}\label{b-c}
Let $(S^t)_{t \in I}$ be a real analytic path in $\mathcal{T}_g$. Then there exist real analytic functions ${\lambda^t_k}: I \rightarrow \mathbb{R}$ such that for each $t \in I$ the sequence $({\lambda^t_k})$ consist of all eigenvalues of $S^t$ (listed with multiplicities, though not in increasing order).
\end{thm}
Each function $\lambda_k^t$ is called a branch of eigenvalues along $S^t$. More precisely
\begin{defn}
Let $\alpha: [0, 1] \to {\mathcal{T}_g}$ be an analytic path. An analytic function ${\lambda_t}: [0, 1] \to \mathbb{R}$ is called a branch of an eigenvalue along $\alpha$ if, for each $t$, ${\lambda_t}$ is an eigenvalue of $\alpha(t)$. If ${\lambda_0}= {\lambda_i}(\alpha(0))$ then we shall say that $\lambda_t$ is a branch of eigenvalues along $\alpha$ that starts as $\lambda_i$. If the underlying path $\alpha$ is fixed then we shall skip referring to it.
\end{defn}
Here, instead of considering $\lambda_1$, we consider branches of eigenvalues that start as $\lambda_1$ and modify question \ref{ques} as:
\begin{que}
For any $g \ge 2$ does there exist branches of eigenvalues in ${\mathcal{T}_g}$ that start as $\lambda_1$ and exceeds $\frac{1}{4}$ eventually ?
\end{que}
Fortunately this modified question turns out to be much easier than the original one and we have a positive answer to it (see theorem \ref{br}).
\subsection{Multiplicity}\label{mult}
For any eigenvalue $\lambda$ of $S$, the dimension of $\ker(\Delta - \lambda.\id)$ is called the multiplicity of $\lambda$. If the multiplicity of $\lambda_1$ were one for all closed hyperbolic surfaces of genus $g$ then theorem \ref{br} would have showed the existence of surfaces with ${\lambda_1} > \frac{1}{4}$ implying conjecture \ref{conj}. However this is not the case and in fact the following is proved in \cite{C-V}:
\begin{thm}
For every $g \geq 3$ and $n\geq0$ there exists a surface $S \in {\mathcal{M}_{g, n}}$ such that ${\lambda_1}(S)$ is small and has multiplicity equal to the integral part of $\frac{1+\sqrt{8g+1}}{2}$.
\end{thm}
For $g \geq 3$ the above bound is more than $3$. Hence our methods in theorem \ref{e2} for $g =2$ do not work for $g \geq 3$. In \cite{O} the following upper bound on the multiplicity of a small eigenvalue is proved
\begin{prop}\label{otal}
Let $S$ be a finite area hyperbolic surface of type $(g, n)$. Then the multiplicity of a small eigenvalue of $S$ is at most $2g-3+n$.
\end{prop}
Our last result is an improvement of this result for hyperbolic surfaces of type $(0, n)$ (see theorem \ref{puncturedsphere}).
\section{results}
As mentioned before, it is known that there are closed hyperbolic surfaces of genus two with $\lambda_1 > \frac{1}{4}$ (in fact with $> 3.8$). Our first result, in some sense, describes how large is the open subset
\begin{equation*}
{\mathcal{B}_2}(\frac{1}{4})= \{ S \in {\mathcal{M}_2}: {\lambda_1}(S) > \frac{1}{4} \}.
\end{equation*}
\begin{thm}\label{e2}
${\mathcal{B}_2}(\frac{1}{4})$ is an unbounded set that disconnects ${\mathcal{M}_2}$.
\end{thm}
\textbf{Sketch of the proof of Theorem \ref{e2}:}\\*
We first prove that ${\mathcal{B}_2}(\frac{1}{4})$ disconnects ${\mathcal{M}_2}$. We argue by contradiction and assume that $\mathcal{M}_2 \setminus {\mathcal{B}_2}(\frac{1}{4})$ is connected. Now for any $S \in \mathcal{M}_2 \setminus {\mathcal{B}_2}(\frac{1}{4})$, ${\lambda_1}(S)$ is small and hence has multiplicity exactly one by \cite{O}. We shall see that, in fact, the {\it nodal set} of the ${\lambda_1}(S)$-eigenfunction (see \S3) consists of simple closed curves. With the help of this property we shall deduce that the nodal set of the first eigenfunction is constant, up to isotopy, on $\mathcal{M}_2 \setminus {\mathcal{B}_2}(\frac{1}{4})$. Finally, using an argument involving geodesic pinching we shall show that there exist surfaces $S_1$ and $S_2$ in ${\mathcal{M}_2} \setminus {\mathcal{B}_2}(\frac{1}{4})$ such the nodal sets of the ${\lambda_1}({S_1})$-eigenfunction is not isotopic to the nodal set of the ${\lambda_1}({S_2})$-eigenfunction. This provides the desired contradiction. The rest of the theorem i.e. ${\mathcal{B}_2}(\frac{1}{4})$ is unbounded is deduced from a description of the components of $\mathcal{M}_2 \setminus {\mathcal{B}_2}(\frac{1}{4})$.

For finite area hyperbolic surfaces with Euler characteristic two the ideas in the above proof carries over to provide the following.
\begin{thm}\label{E2}
For any $(g, n)$ with $2g-2+n=2$ (i.e. $(g, n)= (2, 0), (1, 2)$ or $(0, 4)$) the set ${\mathcal{C}_{g, n}}(\frac{1}{4}) = \{ S \in {\mathcal{M}_{g, n}}: {\lambda_1}(S) \ge \frac{1}{4}\}$disconnects ${\mathcal{M}_{g, n}}$. Moreover for $(g, n) = (2, 0)$ and $(1, 2)$ it is unbounded.
\end{thm}
Our next result is on the existence of branches of eigenvalues in $\mathcal{T}_g$, for any $g \ge 3$, that start as $\lambda_1$ and eventually becomes larger than $\frac{1}{4}$.
\begin{thm}\label{br}
There are branches of eigenvalues in $\mathcal{T}_g$ that start as $\lambda_1$ and take values strictly bigger than $\frac{1}{4}$.
\end{thm}
Recall that $\mathcal{T}_2$ can be embedded in $\mathcal{T}_g$ as an analytic subset containing surfaces with certain symmetries (see \S\ref{brnch}). The branches in theorem \ref{br} will be obtained by composing the branches in $\mathcal{T}_2$ by the above embedding $\Pi: {\mathcal{T}_2} \to {\mathcal{T}_g}$. We shall use a {\it geodesic pinching} argument to prove that among these branches there are ones that start as $\lambda_1$.

Our last result is on the multiplicity of $\lambda_1$ of genus zero hyperbolic surfaces, punctured spheres.
\begin{thm}\label{puncturedsphere}
Let $S$ be a hyperbolic surface of genus $0$. If ${\lambda_1}(S) \leq \frac{1}{4}$ is an eigenvalue then the multiplicity of ${\lambda_1}(S)$ is at most three.
\end{thm}
\textbf{Sketch of proof:} ~ Let $S$ be a hyperbolic surface of genus $0$ with $n$ punctures. Let $\overline{S}$ denote the closed surface obtained by filling in the punctures of $S$. Assume that ${\lambda_1}(S) \leq \frac{1}{4}$ is an eigenvalue. Let $\phi$ be a ${\lambda_1}(S)$-eigenfunction with nodal set $\mathcal{Z}(\phi)$ (\S2) which is a finite graph by \cite{O} (see lemma \ref{O}).

Using \textbf{Jordan curve theorem} and \textbf{Courant's nodal domain theorem} we shall deduce the simple description of $\overline{\mathcal{Z}(\phi)}$ as a simple closed curve in $\overline{S}$. In particular, if one of the punctures $p$ of $S$ lies on $\overline{\mathcal{Z}(\phi)}$ then the number of arcs in $\overline{\mathcal{Z}(\phi)}$ emanating from $p$ is at most two.

Let $p$ be one of the punctures of $S$. It is a standard fact that in any cusp around $p$ any ${\lambda_1}(S)$-eigenfunction $\phi$ has a Fourier development of the form:
\begin{equation}\label{expresss}
\phi(x, y)= {\phi_0}{y^{1-s}} + {\sum_{j \geq 1}}\sqrt{\frac{2jy}{\pi}}{K_{s-\frac{1}{2}}}(jy)({\phi^e_j}\cos(j.x) + {\phi^o_j}\sin(j.x))
\end{equation}
where ${\lambda_1}(S)= s(1-s)$ with $s \in (\frac{1}{2}, 1]$ and $K$ is the modified Bessel function of exponential decay (see \S2). Denote the vector space generated by ${\lambda_1}(S)$-eigenfunctions by $\mathcal{E}_1$ and consider the map $\pi: {\mathcal{E}_1} \to {\mathbb{R}^3}$ given by $\pi(\phi) = ({\phi_0}, {\phi^e_1}, {\phi^o_1})$. This is a linear map and so if $\dim{\mathcal{E}_1} >3$ then $\ker{\pi}$ is non-empty. Let $\psi \in \ker{\pi}$ i.e. ${\psi_0} = {\psi^e_1}= {\psi^o_1}= 0$. Then by the result \cite{Ju} of Judge, the number of arcs in $\overline{\mathcal{Z}(\psi)}$ emanating from $p$ is at least four, a contradiction to the above description of $\overline{\mathcal{Z}(\phi)}$ at $p$.
\section{Preliminaries}
In this section we recall some definitions and results that will be necessary in later sections. Let $S$ be a finite area hyperbolic surface. Then $S$ is homeomorphic to a closed surface with finitely many points removed. Each of these point, called punctures, has special neighborhoods in $S$ called {\it cusp}s.
\subsection{Cusps} Denote by $\iota$ the parabolic isometry $\iota: z \rightarrow z + 2\pi$. For a choice of $t > 0$, a cusp $\mathcal{P}^t$ is the half-infinite cylinder $\{z= x+iy : y > \frac{2\pi}{t} \}/<\iota>$. The boundary curve $\{y = \frac{2\pi}{t} \}$ is a {\it horocycle} of length $t$. The hyperbolic metric on $\mathcal{P}^t$ has the form:
\begin{equation}
d{s^2} = \frac{d{x^2} + d{y^2}}{y^2}.
\end{equation}
Any function $f \in {L^2}(\mathcal{P}^t)$ has a Fourier development in the $x$ variable of the form
\begin{equation}
 f(z) = \sum_{n \in {{\mathbb Z}^*}} {f_n}(y) \cos(nx + {\theta_n}).
\end{equation}
If $f$ satisfy the equation $\Delta f = s(1-s)f$ then the above expression can be simplified as
\begin{equation}\label{expres}
 f(z) = {f_0}(y) + {\sum_{j \geq 1}}{f_j}\sqrt{\frac{2jy}{\pi}}{K_{s-\frac{1}{2}}}(jy)\cos(j.x - {\theta_j})$$$$ ={f_0}(y) + {\sum_{j \geq 1}}\sqrt{\frac{2jy}{\pi}}{K_{s-\frac{1}{2}}}(jy)({f^e_j}\cos(j.x) + {f^o_j}\sin(j.x))
\end{equation}
where $K_s$ is the modified {\it Bessel function} (see \cite{Ju}) and
\begin{equation}
{f_0}(y) = {f_{0, 1}}{y^s} + {f_{0, 2}}{y^{1-s}} ~~ \textrm{if} ~~ s \neq \frac{1}{2} ~~ \textrm{and}$$$$ ~~ {f_0}(y) = {f_{0, 1}}{y^\frac{1}{2}} + {f_{0, 2}}{y^\frac{1}{2}}\log{y} ~~ \textrm{if} ~~ s =\frac{1}{2}.
\end{equation}
The function $f$ is called {\it cuspidal} if ${f_0}(y) \equiv 0$.
\subsection{Nodal sets}
For any function $f: S \rightarrow \mathbb{R}$, the set $\{ x \in S: f(x)= 0 \}$ is called the {\it nodal set} ${\mathcal Z}(f)$ of $f$. Each component of $S \setminus {\mathcal Z}(f)$ is called a {\it nodal domain} of $f$. In a neighborhood of a regular point $p \in {\mathcal Z}(f)$ (${\nabla_p} f \neq 0$) the implicit function theorem implies that ${\mathcal Z}(f)$ is a smooth curve. In a neighborhood of a critical point $p \in {\mathcal Z}(f)$ (${\nabla_p} f = 0$), it is not so simple to describe ${\mathcal Z}(f)$. When $f$ is an eigenfunction of the Laplacian we have the following description due to S. Y. Cheng \cite{Che}:
\begin{thm}\label{cheng}
Let $S$ be a surface with a $C^\infty$ metric. Then, for any solution of the equation $(\Delta + h)\phi =0$, $h \in {C^\infty}(S)$, one has:\\*
$(i)$ Critical points on the nodal set $\mathcal{Z}(\phi)$ are isolated.\\*
$(ii)$ Any critical point in $\mathcal{Z}(\phi)$ has a neighborhood $N$ in $S$ which is diffeomorphic to the disc $\{ z \in \mathbb{C}: |z| < 1\}$ by a $C^1$-diffeomorphism that sends $\mathcal{Z}(\phi) \cap N$ to an equiangular system of rays.
\end{thm}
\begin{rem}\label{rcheng}
In particular, if $p \in \mathcal{Z}(\phi)$ is a critical point of $\phi$ then the degree of the graph $\mathcal{Z}(\phi)$ at $p$ is at least $4$. Hence if a component of $\mathcal{Z}(\phi)$ is a simple closed loop then it is automatically smooth.
\end{rem}
When $S$ is closed theorem \ref{cheng} implies that ${\mathcal Z}(\phi)$ is a finite graph. When $S$ is non-compact with finite area it implies {\it local} finiteness of ${\mathcal Z}(\phi)$ but not {\it global}. In this particular case we have the following lemma due to Jean-Pierre Otal \cite[Lemma 6]{O} (the second part is \cite[Lemma 1]{O})
\begin{lem}\label{O}
Let $S$ be a hyperbolic surface with finite area and let $\phi : S \rightarrow \mathbb{R}$ be a $\lambda$-eigenfunction with $\lambda \leq \frac{1}{4}$. Then the closure of $\mathcal{Z}(\phi)$ in $\overline{S}$ is a finite graph. Moreover, each nodal domain of $\phi$ has negative Euler characteristic.
\end{lem}
In particular, $\overline{{\mathcal Z}(\phi)}$ is a union (not necessarily disjoint) of finitely many {\it cycles} in $\overline{S}$ that may contain some of the punctures of $S$. Next we recall Courant's nodal domain theorem
\begin{thm}\label{courant}
Let $S$ be a closed hyperbolic surface. Then the number of nodal domains of a ${\lambda_i}(S)$-eigenfunction can be at most $i+1$.
\end{thm}
The proof (see \cite{Cha} or \cite{Che}) of this theorem works also for finite area hyperbolic surfaces if ${\lambda_i} < \frac{1}{4}$. In particular, for a hyperbolic surface $S$ with finite area if ${\lambda_1}(S) < \frac{1}{4}$ then the number of nodal domains of a ${\lambda_1}(S)$-eigenfunction is at most two. Since any ${\lambda_1}$-eigenfunction $\phi$ has mean zero, ${\mathcal Z}(\phi)$ must disconnect $S$. Hence any $\lambda_1$-eigenfunction has exactly two nodal domains.
\section{Genus two: Proof of Theorem \ref{e2}}
We begin by proving that ${\mathcal{B}_2}(\frac{1}{4})$ disconnects ${\mathcal{M}_2}$. We argue by contradiction and assume that $\mathcal{M}_2 \setminus {\mathcal{B}_2}(\frac{1}{4})$ is connected. Now, for any $S \in \mathcal{M}_2 \setminus {\mathcal{B}_2}(\frac{1}{4})$: ${\lambda_1}(S) \le \frac{1}{4}$ and so ${\lambda_1}(S)$ is simple by \cite{O}. Hence to a surface $S \in \mathcal{M}_2 \setminus {\mathcal{B}_2}(\frac{1}{4})$ one can assign its first non-constant eigenfunction $\phi_S$ without any ambiguity. We assume that $\phi_S$ is normalized i.e.
\begin{equation}
{\int_S} {\phi_S^2} d{\mu_S} = 1.
\end{equation}
Let $\mathcal{Z}({\phi_S})$ denote the {\it nodal set} of $\phi_S$. Since $\phi_S$ is the first eigenfunction, by Courant's nodal domain theorem, $S \setminus {\mathcal{Z}({\phi_S})}$ has exactly two components. Denote by ${S^{+}}({\phi_S})$ (resp. ${S^{-}}({\phi_S})$) the component of $S \setminus {\mathcal{Z}({\phi_S})}$ where $\phi_S$ is positive (resp. negative). By Euler-Poicar\'{e} formula applied to the cell decomposition of $S$ consisting of nodal domains of $\phi_S$ as the two skeleton and the nodal set $\mathcal{Z}({\phi_S})$ as the one skeleton we have the following equality:
\begin{equation}\label{euler-poincare}
\chi (S) = \chi({S^{+}}({\phi_S})) + \chi({S^{-}}({\phi_S})) + \chi({\mathcal Z}({\phi_S})).
\end{equation}
Since $\chi(S)= -2$ and both $\chi({S^{+}}({\phi_S}))$ and $\chi({S^{-}}({\phi_S}))$ are negative by lemma \ref{O}, we conclude from \eqref{euler-poincare} that $\chi(\mathcal{Z}({\phi_S}))= 0$. This means that $\mathcal{Z}({\phi_S})$ consists of simple closed curve(s) that divide $S$ into exactly two components. Moreover, since no nodal domain of ${\phi_S}$ is a disc or an annulus by lemma \ref{O}, each curve in $\mathcal{Z}({\phi_S})$ is {\it essential} (homotopically non-trivial in $S$) and no two curves in $\mathcal{Z}({\phi_S})$ are homotopic. In particular,
\begin{claim}\label{description}
For any $S \in \mathcal{M}_2 \setminus {\mathcal{B}_2}(\frac{1}{4})$, the nodal set $\mathcal{Z}({\phi_S})$ of $\phi_S$ consists either of tree smooth simple closed curves that divide $S$ into two pair of pants (the first picture below) or of a unique smooth simple closed curve that divides $S$ into two tori with one hole (the second picture below).
\end{claim}

\centerline{\includegraphics[height=3in]{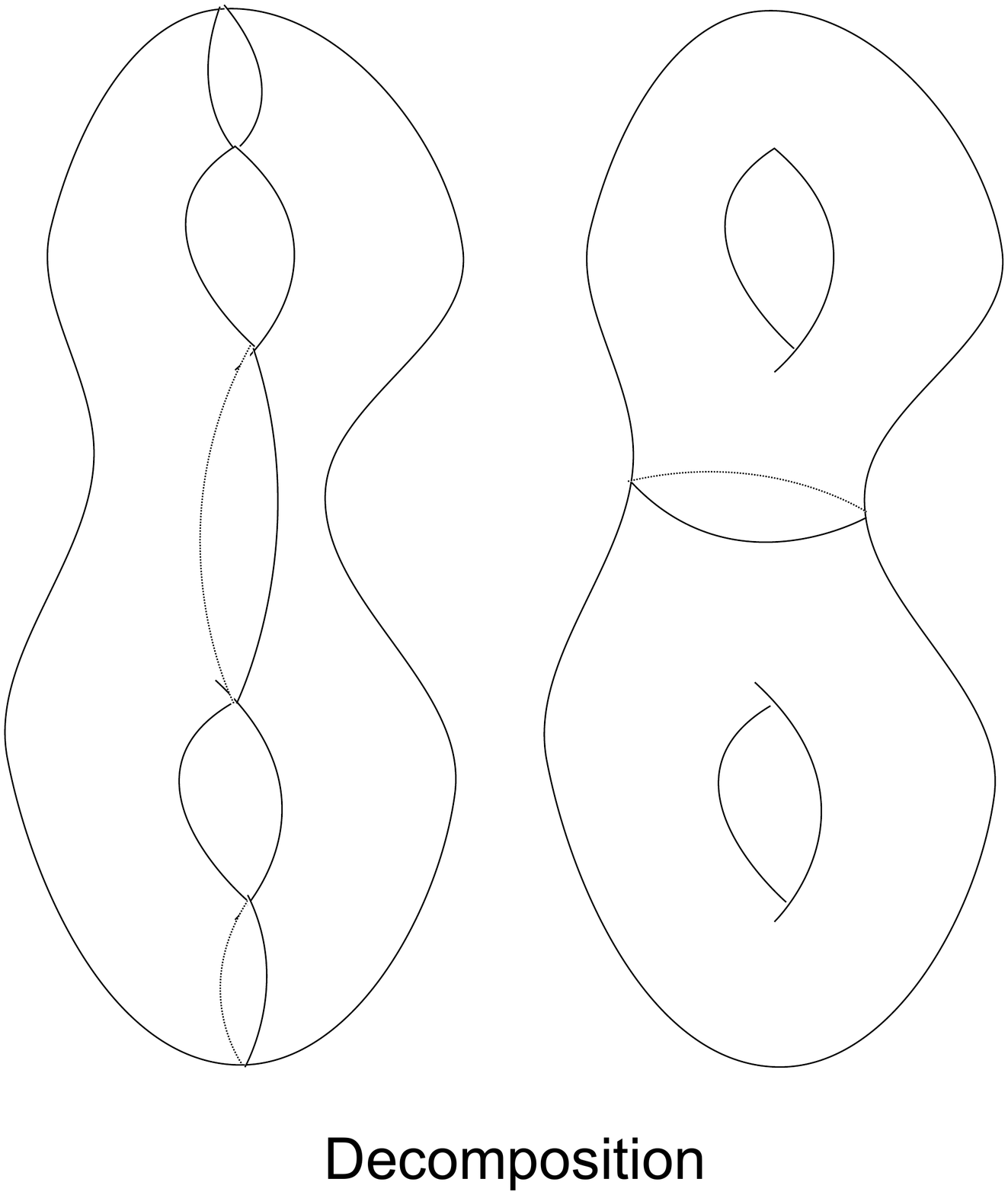}}
Now we have the following:
\begin{claim}\label{isotopy}
Let $S \in {\mathcal{M}_2}$ such that ${\lambda_1}(S)$ is simple and the nodal set $\mathcal{Z}({\phi_S})$ of the ${\lambda_1}(S)$-eigenfunction $\phi_S$ is also simple. Then $S$ has a neighborhood $\mathcal{N}(S)$ in ${\mathcal{M}_2}$ such that for any ${S^{'}} \in \mathcal{N}(S)$ the nodal set $\mathcal{Z}({\phi_{S^{'}}})$ is isotopic to $\mathcal{Z}({\phi_S})$.
\end{claim}
\begin{proof}
First observe that ${\lambda_1}(S)$ being simple we have a neighborhood $\mathcal{N'}(S)$ in ${\mathcal{M}_2}$ such that for any $S' \in \mathcal{N}(S)$ ${\lambda_1}(S')$ is simple. Hence $\phi_S'$ is well defined too.

Now $\phi_S$ is the ${\lambda_1}(S)$-eigenfunction, so $S \setminus \mathcal{Z}({\phi_S})$ has exactly two connected components $S^+$ and $S^-$ such that $\phi_S$ has positive sign on $S^+$. So necessarily $\phi_S$ has negative sign on $S^-$. Now consider a tubular neighborhood $\mathcal{T}_S$ of $\mathcal{Z}({\phi_S})$. By \cite[Theorem 3.36]{M}(see also \cite{H}, \cite{Ji}) we have a neighborhood $\mathcal{N}(S) \subset \mathcal{N'}(S)$ of $S$ such that for any $S' \in \mathcal{N}(S)$, $\phi_{S'}$ has positive sign on ${S^+} \setminus {\mathcal{T}_S}$ and negative sign on ${S^-} \setminus {\mathcal{T}_S}$. In particular, $\mathcal{Z}(\phi_{S'}) \subset {\mathcal{T}_S}$. Hence by the description of $\mathcal{Z}(\phi_{S'})$ as in claim \ref{description} the proof follows.
\end{proof}
Therefore, there exists $S \in {\mathcal{M}_2} \setminus {\mathcal{B}_2}(\frac{1}{4})$ such that $\mathcal{Z}({\phi_S})$ consists of only one curve if and only if for all ${S^{'}} \in {\mathcal{M}_2} \setminus {\mathcal{B}_2}(\frac{1}{4})$, $\mathcal{Z}({\phi_{S^{'}}})$ consists of only one curve. This is a contradiction to proposition \ref{diffisotopy}.
\begin{defn}
The {\it systole $s(S)$} of a surface $S$ is the minimum of the lengths of closed geodesics on $S$. The {\it injectivity radius} of $S$ at a point $p$ is the maximum of the radius of the geodesic discs with center $p$ that embed in $S$. For any $\epsilon>0$ the set of points of $S$ with injectivity radius at least $\epsilon$ is denoted by $S^{[\epsilon, \infty)}$. Each point in the complement $S^{(0, \epsilon)}= S \setminus S^{[\epsilon, \infty)}$ has injectivity radius at most $\epsilon$. $S^{[\epsilon, \infty)}$ and $S^{(0, \epsilon)}$ are respectively called {\it $\epsilon$-thick part} and {\it $\epsilon$-thin part} of $S$.
\end{defn}
\begin{prop}\label{diffisotopy}
Let $S$ be a finite area hyperbolic surface of type $(g, n)$. Let $G = (\gamma_i)_{i=1}^k$ be a collection of smooth, mutually non-intersection simple closed curves on $S$ that separates $S$ in exactly two components. Assume that $G$ is minimal in the sense that no proper subset of $G$ can separate $S$. Then given any $\epsilon, \delta >0$ there exists a finite area hyperbolic surface $S_G$ of type $(g, n)$ with $s(S_G) < \epsilon$ such that $\lambda_1(S_G)< \delta$ is simple and the nodal set of the $\lambda_1(S_G)$-eigenfunction is isotopic to $G$.
\end{prop}
\begin{rem}
It is not very difficult to construct two collections of curves on $S$, as in the above lemma, that are not isotopic. In particular for $(g, n)= (2, 0)$ claim \ref{description} provides two such collections. Therefore the above lemma indeed provide two surfaces $S_1$ and $S_2$ in $\mathcal{M}_2$ such that $S_1$, $S_2 \in \mathcal{M}_2 \setminus {\mathcal{B}_2}(\frac{1}{4})$ and $\mathcal{Z}(\phi_{S_1})$ is not isotopic to $\mathcal{Z}(\phi_{S_2})$.
\end{rem}
Proof of Proposition \ref{diffisotopy} uses the behavior of sequences of small eigenpairs over degenerating sequences of hyperbolic surfaces. For precise definitions of these concepts we refer the reader to \cite{M}.
\begin{proof} Without loss of generality we may assume that each curve in $G$ is a geodesic. Extend $G$ to a pants decomposition $P = (\gamma_i)_{i=1}^{3g-3+n}$ of $S$. Let $(l_i, \theta_i)$ denote the Fenchel-Nielsen coordinates on $\mathcal{T}_{g, n}$ with respect to $(\gamma_i)_{i=1}^{3g-3+n}$. Here $l_i$ denotes the length parameter and $\theta_i$ denotes the twist parameter along $\gamma_i$.

Now consider the sequence of surfaces $({S_m})$ in ${\mathcal{T}_{g, n}}$ such that $l_i(S_m) = \frac{1}{m}$ for $i \le k$, $l_j = c_1 > 0$ for $j > k$ and $\theta_j = c_2 >0$ for $1 \le j \le 3g-3+n$. Then, up to extracting a subsequence, $({S_m})$ converges to a finite area hyperbolic surface ${S_\infty} \in \partial{\mathcal{M}_{g, n}}$. Let us denote the extracted subsequence by $(S_m)$ it self. Observe that $S_\infty$ is obtained from $S$ by pinching the geodesics in $G$. Namely, for each $i=1, ..., k$ there is a geodesic $\gamma^m_i$ in $S_m$, in the homotopy class of $\gamma_i$, whose length tends to zero as $m \to \infty$.

The number of components of ${S_\infty} \in \overline{\mathcal{M}_{g, n}}$ is exactly two. Hence by \cite{C-C}, ${\lambda_1}({S_m}) \to 0$ and all other eigenvalues of $S_m$ stay away from zero. In particular ${\lambda_1}({S_m})$ is simple for $m$ sufficiently large. Let $\phi_{S_m}$ be the $\lambda_1(S_m)$-eigenfunction with $L^2$-norm $1$. Recall that we want to prove that for any $\epsilon, \delta>0$ there exists a $S_G$ with $s(S_G) < \epsilon$ such that $\lambda_1(S_G) < \delta$ is simple and the nodal set of the $\lambda_1(S_G)$-eigenfunction is isotopic to $G$. Since $s(S_m) \to 0$ by construction and ${\lambda_1}({S_m}) \to 0$ by above it suffices to prove that $\mathcal{Z}(\phi_{S_m})$ is isotopic to $G$ for sufficiently large $m$.

Now we apply \cite[Theorem 3.34]{M} to extract a subsequence of $\phi_{S_m}$ that converges uniformly over compacta to a $0$-eigenfunction $\phi_\infty$ of $S_\infty$ with $L^2$-norm $1$. Let us denote the extracted subsequence by $(S_m)$ itself. Since $0$-eigenfunctions are constant functions, $\phi_\infty$ is constant on each components of $S_\infty$.
\begin{claim}\label{sign}
The two constant values of $\phi_\infty$ on the two components of $S_\infty$ are non-zero and have opposite sign.
\end{claim}
\begin{proof}
For $\epsilon>0$ let us denote the $L^2$-norm of $\phi_{S_m}$ restricted to $S_m^{(0, \epsilon)}$ by $\lVert {\phi_{S_m}} \rVert_{S_m^{(0, \epsilon)}}$. By the uniform convergence of $\phi_{S_m}$ to $\phi_\infty$ over compacta we have
\begin{equation*}
{\int_{S_\infty^{[\epsilon, \infty)}}} {\phi^2_\infty} = {\lim_{m \to \infty}} {\int_{S_m^{[\epsilon, \infty)}}} {\phi^2_{S_m}} = 1 - {\lim_{m \to \infty}} \lVert {\phi_{S_m}} \rVert^2_{S_m^{(0, \epsilon)}}.
\end{equation*}
Since ${\int_{S_\infty}} {\phi^2_\infty} = {\lim_{\epsilon \to 0}}{\int_{S_\infty^{[\epsilon, \infty)}}} {\phi^2_\infty} = 1$ we obtain that for any $\delta>0$ there exists $\epsilon>0$ such that ${\lim_{m \to \infty}}\lVert {\phi_{S_m}} \rVert_{S_m^{(0, \epsilon)}} \le \delta$. Now
\begin{equation*}
|{\int_{S_\infty^{[\epsilon, \infty)}}} {\phi_\infty}| = {\lim_{m \to \infty}} |{\int_{S_m^{[\epsilon, \infty)}}} {\phi_{S_m}}| = |0 - {\lim_{m \to \infty}} {\int_{S_m^{(0, \epsilon)}}} {\phi_{S_m}}|$$$$ \le {\lim_{m \to \infty}}\sqrt{|S_m^{(0, \epsilon)}|} \lVert {\phi_{S_m}} \rVert_{S_m^{(0, \epsilon)}} (\text{ by Holder inequality}) \le \delta{\lim_{m \to \infty}}\sqrt{|S_m^{(0, \epsilon)}|}.
\end{equation*}
Here $|S_m^{(0, \epsilon)}|$ denotes the area of $S_m^{(0, \epsilon)}$. Recall that, for any $m \in \mathbb{N} \cup {\infty}$, ${\lim_{\epsilon \to 0}}|S_m^{(0, \epsilon)}| = 0$. So for $m \ge 1$ and $\epsilon$ sufficiently small:
\begin{equation*}
|{\int_{S_\infty^{[\epsilon, \infty)}}} {\phi_\infty}| < \delta ~~ \text{and} ~~ |S_m^{(0, \epsilon)}| < \delta.
\end{equation*}
Finally, taking $\epsilon$ to be sufficiently small, we calculate:
\begin{equation*}
 |{\int_{S_\infty}} {\phi_\infty}| \le |{\int_{S_\infty^{[\epsilon, \infty)}}} {\phi_\infty}| + |{\int_{S_\infty^{(0, \epsilon)}}} {\phi_\infty}| \le \delta + \sqrt{|S_\infty^{(0, \epsilon)}|} \lVert {\phi_{S_\infty}} \rVert_{S_\infty^{(0, \epsilon)}} \le 2\delta
\end{equation*}
since $\lVert {\phi_{S_\infty}} \rVert_{S_\infty^{(0, \epsilon)}} < \lVert {\phi_{S_\infty}} \rVert = 1.$ Since $\delta$ is arbitrary we conclude that ${\int_{S_\infty}} {\phi_\infty} = 0$. Hence ${\phi_\infty}$ has $L^2$-norm $1$ and mean zero.

Since $\phi_\infty$ has $L^2$-norm $1$ at least one of the two constant values of $\phi_\infty$ on the two components of $S_\infty$ is non-zero. Since $\phi_\infty$ has mean zero both of these values are non-zero have opposite sign.
\end{proof}
As the length of $\gamma^m_i$ tends to zero, we may assume that the collar neighborhood $C^m_i$ of $\gamma^m_i$ with two boundary components of length $1$ embeds in $S_m$ and $(C^m_i)_{i=1}^k$ are mutually disjoint. At this point we recall that $G$ is minimal in the sense that no proper subset of $G$ can separate $S$. Hence not only $S_m \setminus {\cup_{i=1}^k}(C^m_i)$ separates $S$ in exactly two components but also no proper sub-collection of $(C^m_i)_{i=1}^k$ can separate $S_m$. In particular, for each $i$, the limits of the two components of $\partial{C^m_i}$ belong to two different components of $S_\infty$. Using claim \ref{sign} let us denote the limits of these two boundary sets by $B_i^\infty(+)$ and $B_i^\infty(-)$ such that ${\phi_\infty}|_{B_i^\infty(+)} > 0$ and ${\phi_\infty}|_{B_i^\infty(-)} < 0$. Correspondingly denote the two components of $\partial{C^m_i}$ by $B_i^m(+)$ and $B_i^m(-)$ such that $B_i^\infty(\pm)$ is the limit of $B_i^m(\pm)$ respectively. By the uniform convergence of $\phi_{S_m}$ to $\phi_\infty$ over compacta we conclude that, for sufficiently large $m$, ${\phi_{S_m}}|_{B_i^m(+)} > 0$ and ${\phi_{S_m}}|_{B_i^m(-)} < 0$. Hence, for $m$ sufficiently large, at least one component of $\mathcal{Z}({\phi_{S_m}})$ is contained in $C^m_i$. Let $Z_i$ denote the union of the components of $\mathcal{Z}({\phi_{S_m}})$ that are contained in $C^m_i$.

Let $\alpha$ be a simple closed loop in $Z_i$. Since $\pi_1(C^m_i)$ is $\mathbb{Z}$ there are only two possibilities for $\alpha$. Either it bounds a disc in $C^m_i$ or it is homotopic to $\gamma^m_i$. Since $\lambda_1(S_m)$ is small, each component of $S_m \setminus \mathcal{Z}({\phi_{S_m}})$ has negative Euler characteristic by lemma \ref{O}. This discards the possibility that $\alpha$ bounds a disc in $C^m_i$. Hence $\alpha$ is homotopic to $\gamma^m_i$. Let $\beta$ be another simple closed loop in $Z_i$. Then $\beta$ is also homotopic to $\gamma^m_i$ implying that one of the components of $S_m \setminus \mathcal{Z}({\phi_{S_m}})$ has non-negative Euler characteristic. This leaves us with the observation that each $C^m_i$ contains exactly one loop $\alpha^m_i$ from $\mathcal{Z}({\phi_{S_m}})$. By remark \ref{rcheng} $\alpha^m_i$ is in fact smooth. Therefore we have an isotopy of $S$ that sends $\alpha^m_i$ to $\gamma^m_i$. Combining these isotopies we obtain that $\mathcal{Z}({\phi_{S_m}})$ is isotopic to $(\gamma^m_i)_{i=1}^k$.
\end{proof}
It remains to show that ${\mathcal{B}_2}(\frac{1}{4})$ is unbounded. We argue by contradiction and assume that ${\mathcal{B}_2}(\frac{1}{4})$ is bounded. Then we have $\epsilon>0$ such that ${\mathcal{B}_2}(\frac{1}{4})$ is contained in the compact set $\mathcal{I}_\epsilon = \{S \in {\mathcal{M}_2}: s(S) \ge \epsilon\}$ \cite{B}. Now applying lemma \ref{diffisotopy} obtain $S_1$ and $S_2$ in ${\mathcal{M}_2}$ such that $s(S_i)< \epsilon$, $\lambda_1(S_i) < \frac{1}{4}$ is simple and the nodal set of the $\lambda_1(S_1)$-eigenfunction is not isotopic to $\lambda_1(S_2)$-eigenfunction. Since ${\mathcal{M}_2} \setminus \mathcal{I}_\epsilon$ is path connected (see lemma \ref{pconn}) we may have a path $\beta$ in ${\mathcal{M}_2} \setminus \mathcal{I}_\epsilon$ that joins $S_1$ and $S_2$. Then lemma \ref{isotopy} implies that the nodal set of the $\lambda_1(S_1)$-eigenfunction is isotopic to $\lambda_1(S_2)$-eigenfunction. This is a contradiction to our choice of $S_1$ and $S_2$.
\subsection{Proof of Theorem \ref{E2}}
The case $(g, n)= (2, 0)$ follows from the above theorem. It remains to show theorem \ref{E2} for $(g, n)= (1, 2)$ and $(0, 4)$. \textbf{For the rest of the proof we refer to the pair $(g, n)$ for only these two cases.} We argue by contradiction and assume that $\mathcal{M}_{g, n} \setminus {\mathcal{C}_{g, n}}(\frac{1}{4})$ is connected. By definition $\lambda_1(S) < \frac{1}{4}$ for any $S \in \mathcal{M}_{g, n} \setminus {\mathcal{C}_{g, n}}(\frac{1}{4})$. Hence $\lambda_1(S)$ is an eigenvalue and by \cite{O-R} it is the only non-zero small eigenvalue of $S$. So we can consider the first non-constant eigenfunction $\phi_S$ of $S$. As before let $\mathcal{Z}(\phi_S)$ be the nodal set of $\phi_S$. Denote by $\overline{S}$ the surface obtained from $S$ by filling in its punctures and by $\overline{\mathcal{Z}(\phi_S)}$ the closure of $\mathcal{Z}(\phi_S)$ in $\overline{S}$. By lemma \ref{O} $\overline{\mathcal{Z}(\phi_S)}$ is a finite graph. Now apply Euler-Poincar\'{e} formula to the cell decomposition of $\overline{S}$ defined as follows: the punctures on $S$ that do not lie on $\overline{\mathcal{Z}(\phi_S)}$ is the zero skeleton, $\overline{\mathcal{Z}(\phi_S)}$ is the one skeleton and $S \setminus \overline{\mathcal{Z}(\phi_S)}$ is the two skeleton. If $k$ is the number of punctures of $S$ that do not lie on $\overline{{\mathcal Z}({\phi_S})}$ then
\begin{equation}\label{euler-poincare}
\chi (\overline{S}) - k = \chi(S \setminus \overline{\mathcal{Z}(\phi_S)}) + \chi(\overline{\mathcal{Z}({\phi_S})}).
\end{equation}
By lemma \ref{O} each component of $S \setminus \overline{\mathcal{Z}(\phi_S)}$ has negative Euler characteristic and so $\chi(S \setminus \overline{\mathcal{Z}(\phi_S)}) \le -2$. For $(g, n) = (1, 2)$, $\chi (\overline{S})= 0$ and so we have the only possibility $k =2$ and $\chi({\overline{\mathcal Z}({\phi_S})})=0$. For $(g, n) = (0, 4)$, $\chi (\overline{S})= 2$ leaving us with the only possibility $k=4$ and $\chi({\overline{\mathcal Z}({\phi_S})})=0$. Hence none of the punctures of $S$ lie on the closure of the nodal set $\overline{{\mathcal Z}({\phi_S})}$ i.e. $\overline{{\mathcal Z}({\phi_S})}= {\mathcal Z}({\phi_S})$ is a compact subset of $S$. Since $\chi(\overline{\mathcal Z({\phi_S})})=0$ we conclude that ${\mathcal Z}({\phi_S})$ is a union of simple closed curves. Also by lemma \ref{O} we know that no loop in ${\mathcal Z}({\phi_S})$ can bound a disc and no two components of ${\mathcal Z}({\phi_S})$ can be homotopic. Summarizing these observations we get:
\begin{claim}\label{ne2}
Let $S \in \mathcal{M}_{g, n} \setminus {\mathcal{C}_{g, n}}(\frac{1}{4})$. \\*
$(i)$ If $(g, n) = (1, 2)$ then $\mathcal{Z}(\phi_S)$ consists of either exactly one simple closed curve or two simple closed curves. In the first case $\mathcal{Z}(\phi_S)$ divides $S$ into two components one of which is a surface of genus one with a copy of $\mathcal{Z}(\phi_S)$ as its boundary and the other one is a twice punctured sphere with a copy of $\mathcal{Z}(\phi_S)$ as its boundary. In the last case $\mathcal{Z}(\phi_S)$ divides $S$ into two components each of which is a once punctured sphere with two boundary components coming from $\mathcal{Z}(\phi_S)$.\\*
$(ii)$ If $(g, n)= (0, 4)$ then $\mathcal{Z}(\phi_S)$ consists of exactly one simple closed curve (there are two possibilities for this up to isotopy) that separates $S$ into two components each of which is a twice punctured sphere with one boundary component coming from $\mathcal{Z}(\phi_S)$.
\end{claim}
Next we have the following modified version of claim \ref{isotopy}. Let $S \in \mathcal{M}_{g, n} \setminus {\mathcal{C}_{g, n}}(\frac{1}{4})$ with $\phi_S$ the $\lambda_1(S)$-eigenfunction.
\begin{claim}\label{isotopy1}
There exists a neighborhood $\mathcal{N}(S)$ of $S$ in $\mathcal{M}_{g, n}$ such that for any $S' \in \mathcal{N}(S)$: $\lambda_1(S')$ is simple and the nodal set $\mathcal{Z}(\phi_{S'})$ of the $\lambda_1(S')$-eigenfunction $\phi_{S'}$ is isotopic to $\mathcal{Z}(\phi_S)$.
\end{claim}
\begin{proof}
Since $\lambda_1(S)$ is $< \frac{1}{4}$, $\lambda_1$ is a continuous function in a neighborhood of $S$ by \cite{H}(see also \cite{C-C}, \cite{M}) and so we have a neighborhood $\mathcal{N'}(S)$ of $S$ which is contained in $\mathcal{M}_{g, n} \setminus {\mathcal{C}_{g, n}}(\frac{1}{4})$. In particular, $\phi_{S'}$ is well-defined for $S' \in \mathcal{N'}(S)$ and $\mathcal{Z}(\phi_{S'})$ has the description in claim \ref{ne2}. Now consider a tubular neighborhood $\mathcal{T}_S$ of $\mathcal{Z}(\phi_S)$ in $S$ such that $\partial{\mathcal{T}_S}$ has two components $\partial{\mathcal{T}^+_S}$ and $\partial{\mathcal{T}^-_S}$ each of which is a simple closed curve with ${\phi_S}|_{\partial{\mathcal{T}^+_S}} >0$ and ${\phi_S}|_{\partial{\mathcal{T}^-_S}} <0$.

Now $\lambda_1 < \frac{1}{4}$ and simple on $\mathcal{N'}(S)$. Hence by \cite{H} for any compact subset $K$ of $S$ the map: $\Phi: K \times \mathcal{N'}(S) \to \mathbb{R}$ given by $\Phi(x, S')= \phi_{S'}(x)$ is continuous. Considering $K = \partial{\mathcal{T}_S}$ we obtain $\mathcal{N}(S) \subset \mathcal{N'}(S)$ such that for any $S' \in \mathcal{N}(S)$: ${\phi_{S'}}|_{\partial{\mathcal{T}^+_S}} >0$ and ${\phi_{S'}}|_{\partial{\mathcal{T}^-_S}} <0$. In particular, for any $S' \in \mathcal{N}(S)$: $\mathcal{Z}(\phi_{S'})$ has a component inside $\mathcal{T}_S$. Hence by the description of $\mathcal{Z}(\phi_{S'})$ in claim \ref{ne2} we obtain the claim.
\end{proof}
Since by our assumption $\mathcal{M}_{g, n} \setminus {\mathcal{C}_{g, n}}(\frac{1}{4})$ is connected the above claim implies that only one of the two possibilities in claim \ref{ne2} can actually occur. This is a contradiction to proposition \ref{diffisotopy}.

Now we show that ${\mathcal{C}_{1, 2}}(\frac{1}{4})$ is unbounded. We argue by contradiction and assume that ${\mathcal{C}_{1, 2}}(\frac{1}{4})$ is bounded. Then we have $\epsilon>0$ such that ${\mathcal{C}_{1, 2}}(\frac{1}{4})$ is contained in the compact set $\mathcal{I}_\epsilon = \{S \in {\mathcal{M}_{1, 2}}: s(S) \ge \epsilon\}$ \cite{B}. Applying lemma \ref{diffisotopy} we obtain $S_1$ and $S_2$ in ${\mathcal{M}_{1, 2}}$ such that $s(S_i)< \epsilon$, $\lambda_1(S_i) < \frac{1}{4}$ is simple and the nodal set of the $\lambda_1(S_1)$-eigenfunction is not isotopic to $\lambda_1(S_2)$-eigenfunction. Since ${\mathcal{M}_{1, 2}} \setminus \mathcal{I}_\epsilon$ is path connected (see lemma \ref{pconn}) we may have a path $\beta$ in ${\mathcal{M}_{1, 2}} \setminus \mathcal{I}_\epsilon$ that joins $S_1$ and $S_2$. Then lemma \ref{isotopy1} implies that the nodal set of the $\lambda_1(S_1)$-eigenfunction is isotopic to $\lambda_1(S_2)$-eigenfunction. This is a contradiction to our choice of $S_1$ and $S_2$.

\centerline{\includegraphics[height=4in]{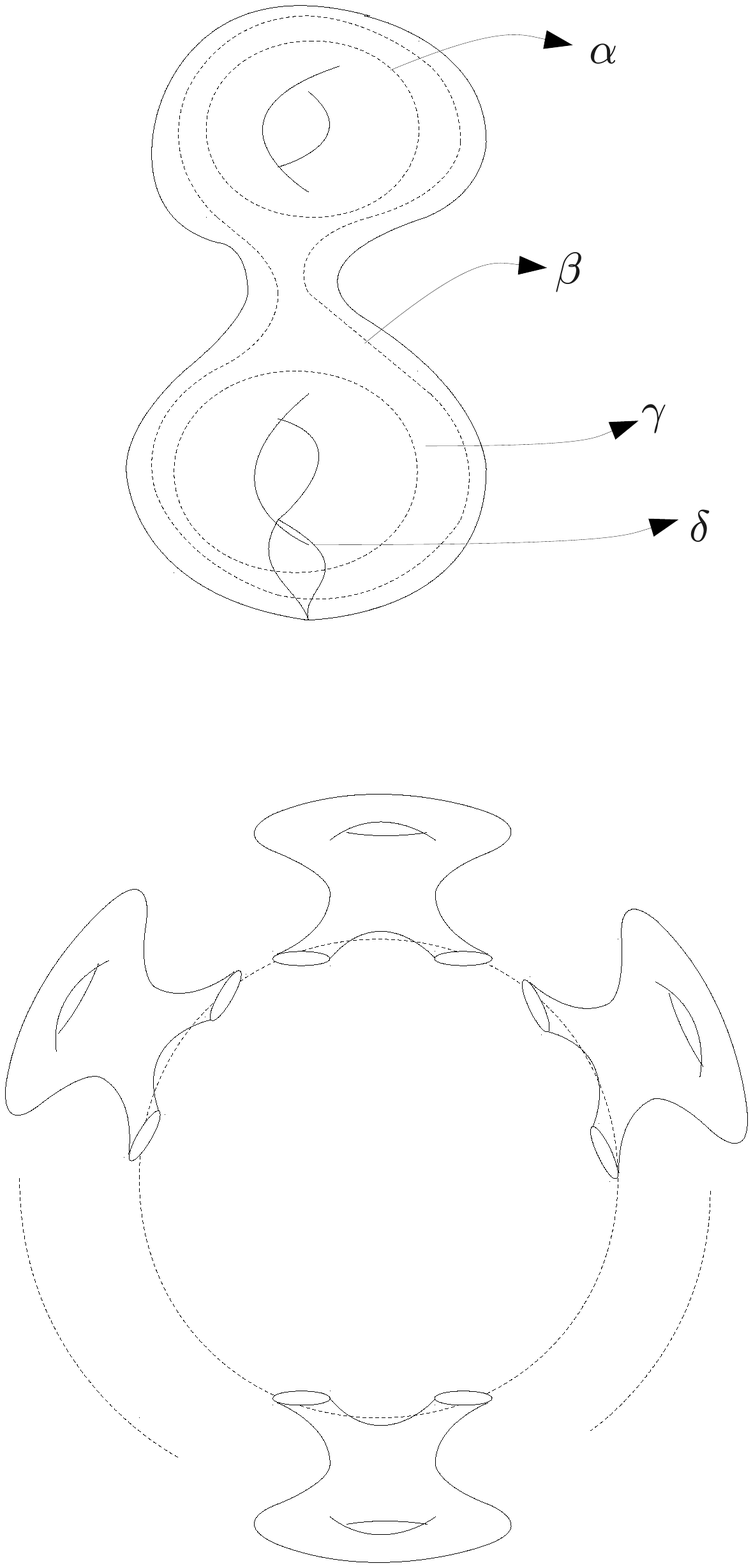}}

\section{Branches of eigenvalues}\label{brnch}
In this section we consider branches of eigenvalues along paths in $\mathcal{T}_g$. Main purpose of doing so is that the multiplicity of $\lambda_i$, in particular $\lambda_1$ is not one in general. Therefore along 'nice' paths in $\mathcal{T}_g$ the functions $\lambda_i$ may not be 'nice' enough (see introduction). However, theorem \ref{b-c} shows that up to certain choice at points of multiplicity $\lambda_i$'s are in fact 'nice'. This 'nice' choice makes $\lambda_i$ into a branch of eigenvalues. Theorem \ref{br} says that if we restrict ourselves to branches of eigenvalues then we have a positive answer to conjecture \ref{conj}, namely there are branches of eigenvalues that start as $\lambda_1$ and becomes more than $\frac{1}{4}$.
\begin{proof}[Proof of Theorem \ref{br}]
We begin by explaining the the embedding $\Pi : {\mathcal{T}_2} \to {\mathcal{T}_g}$ (see the next figure). Let $S$ be the closed hyperbolic surface of genus two and $\alpha, \beta, \gamma$, $\delta$ are four geodesics on $S$ as in the following picture. Now cut $S$ along $\delta$ to obtain a hyperbolic surface $S^*$ with genus one and two geodesic boundaries (each a copy of $\delta$). Consider $g-1$ many copies of $S^*$ and glue them along their consecutive boundaries after arranging them along a circle as in the picture below. Let $\Pi(S)$ denote the resulting hyperbolic surface.

Now take a geodesic pants decomposition $(\xi_i)_{i=1, 2, 3}$ of $S$ involving $\delta= {\xi_3}$ and consider the Fenchel-Nielsen coordinates $({l_i}, {\theta_i})_{i=1, 2, 3}$ on $\mathcal{T}_2$ with respect to this pants decomposition. Here ${l_i} = l(\xi_i)$ is the length of the closed geodesic $\xi_i$ and ${\theta_i}$ is the twist parameter at $\xi_i$. The images of $(\xi_i)_{i=1, 2, 3}$ in $\Pi(S)$, $(\xi_i^j)_{i=1, 2, 3; j=1, 2, ..., g-1}$ is a geodesic pants decomposition of $\Pi(S)$. Consider the the Fenchel-Nielsen coordinates $({l_i^j}, {\theta_i^j})_{i=1, 2, 3; j=1, 2,..., g-1}$ on $\mathcal{T}_g$ with respect to this pants decomposition. As before, ${l_i^j} = l(\xi_i^j)$ is the length of the closed geodesic $\xi_i^j$ and ${\theta_i^j}$ is the twist parameter at $\xi_i^j$. With respect to these pants decompositions $\Pi$ is expressed as
\begin{equation}
({l_1}, {l_2}, {l_3}, {\theta_1}, {\theta_2}, {\theta_3}) \to ({\underbrace{{l_1}, {l_2}, {l_3}, {\theta_1}, {\theta_2}, {\theta_3}}_{1}},..., {\underbrace{{l_1}, {l_2}, {l_3}, {\theta_1}, {\theta_2}, {\theta_3}}_{g-1}}).
\end{equation}
This is an analytic map and the image $\Pi(S)$ of any $S \in {\mathcal{T}_2}$ has an isometry $\tau$ of order $(g-1)$ that sends one $6$-tuple $({l_1}, {l_2}, {l_3}, {\theta_1}, {\theta_2}, {\theta_3})$ to the next one. Also $\Pi(S)/\tau$ is isometric to $S$ i.e. $\Pi(S)$ is a $(g-1)$ sheeted covering of $S$. Hence each eigenvalue of $S$ is also an eigenvalue of $\Pi(S)$. In particular, a branch ${\lambda_t}$ of eigenvalues in ${\mathcal{T}_2}$ along $\eta(t)$ is a a branch of eigenvalues in ${\mathcal{T}_g}$ along $\Pi(\eta(t))$.

To finish the proof we need only to find $S \in {\mathcal{T}_2}$ such that ${\lambda_1}(S) = {\lambda_1}(\Pi(S))$. Once we find such a $S$, we can consider any analytic path $\eta$ in ${\mathcal{T}_2}$ such that $\eta(o)=S$ and ${\lambda_1}(\eta(1))> \frac{1}{4}$. Then the branch of eigenvalues ${\lambda_t} = {\lambda_1}(\eta(t))$ along $\Pi(\eta(t))$ would be a branch that we seek.

To show this we employ the technique in claim \ref{diffisotopy}. Let $S_n$ be a sequence of surfaces of genus two on which the lengths of the geodesics $\alpha, \beta$ and $\gamma$ tends to zero. In particular, ${S_n} \to {S_\infty} \in {\mathcal{M}_{0, 3}} \cup {\mathcal{M}_{0, 3}}$ implying ${\lambda_1}(S_n) \to 0$ and ${\lambda_2}(S_n) \nrightarrow 0$. The sequence $\Pi(S_n)$ converges to a surface in ${\mathcal{M}_{0, g+1}} \cup {\mathcal{M}_{0, g+1}}$ and so ${\lambda_1}(\Pi(S_n)) \to 0$ and ${\lambda_2}(\Pi(S_n)) \nrightarrow 0$. So for large $n$, ${\lambda_1}(S_n)< {\lambda_2}(\Pi(S_n))$ implying ${\lambda_1}(S_n) = {\lambda_1}(\Pi(S_n))$.
\end{proof}
\section{Punctured spheres}
We begin this section by recapitulating the ideas in \cite{BBD}. By purely number theoretic methods Atle Selberg showed that for any congruence subgroup $\Gamma$ of SL$(2, \mathbb{Z})$, ${\lambda_1}(\mathbb{H}/\Gamma) \geq \frac{3}{16}$. The purpose in \cite{BBD} was to construct explicit closed hyperbolic surfaces with $\lambda_1$ close to $\frac{3}{16}$. To achieve this goal the authors of \cite{BBD} considered principal congruence subgroups $\Gamma_n$ (see introduction) and corresponding finite area hyperbolic surfaces $\mathbb{H}/{\Gamma_n}$. Then they replaced the cusps in $\mathbb{H}/{\Gamma_n}$, which is even in number, by closed geodesics of small length $t$ and glued them in pairs (see \cite{BBD} for details). The surface $S_t$ obtained in this way is closed, their genus $g$ is independent of $t$ and as $t \to 0$, ${S_t} \to \mathbb{H}/{\Gamma_n}$ in the compactification of the moduli space $\mathcal{M}_g$. Rest of the proof showed that $\lambda_1$ is lower semi-continuous over the family $S_t$. This approach together with the result of Kim and Sarnak provides theorem \ref{aprox}.

Limiting properties of eigenvalues over degenerating family of hyperbolic metrics have been studied well in the literature (to name a few Denis Hejhal \cite{H}, Gilles Courtois-Bruno Colbois \cite{C-C}, Lizhen Ji \cite{Ji}, Scott Wolpert \cite{Wo}, Chris Judge \cite{J}) (see also \cite[Theorem 2]{M}). These limiting results can be summarized as:
\begin{thm}
Let $(S_m)$ be a sequence of hyperbolic surfaces in $\mathcal{M}_{g, n}$ that converges to a finite area hyperbolic surface $S \in \partial{\mathcal{M}_{g, n}}$. Let $({\lambda_m}, {\phi_m})$ be an eigenpair of $S_m$ such that ${\lambda_m} \to \lambda< \infty$. Then, up to extracting a subsequence and up to rescalling, the sequence $(\phi_m)$ converges to a generalized eigenfunction over compacta if one of the following is true \\*
$(i)$ $n=0$ (\cite{Ji}) $(ii)$ $n \neq 0$ and $\lambda <\frac{1}{4}$ (\cite{H}, \cite{C-C}) $(iii)$ $n \neq 0$ and $\lambda >\frac{1}{4}$ (\cite{Wo}) $(iii)$ $n \neq 0$, ${\lambda_m} \leq \frac{1}{4}$ and $\phi_m$ is cuspidal (\cite{M}).
\end{thm}
Recall that there is a copy of ${\mathcal{M}_{0, 2g+n}}$ in the compactification $\overline{\mathcal{M}_{g, n}}$ of ${\mathcal{M}_{g, n}}$. The ideas in \cite{BBD} along with above limiting results imply
\begin{lem}\label{bbd}
For any pair $(g, n)$, ${\Lambda_1}(g, n) \geq {\Lambda_1}(0, 2g+n).$
\end{lem}
Motivated by this we focus on ${\Lambda_1}(0, n)$. Although we would not be able to prove conjecture \ref{conj} we have theorem \ref{puncturedsphere} on the multiplicity of $\lambda_1$ which we prove now.
\subsection{Proof of Theorem \ref{puncturedsphere}}
Let $S$ be a hyperbolic surface of genus $0$ and assume that ${\lambda_1}(S) \leq \frac{1}{4}$ is an eigenvalue. Let $\phi$ be a ${\lambda_1}(S)$-eigenfunction. Then the closure $\overline{\mathcal{Z}(\phi)}$ of the nodal set $\mathcal{Z}(\phi)$ of $\phi$ is a finite graph in $\overline{S}$ by theorem \ref{otal}. In particular, $\overline{\mathcal{Z}(\phi)}$ is a union of closed loops in $\overline{S}$. Observe also that the number of components of $\overline{S} \setminus \overline{\mathcal{Z}(\phi)}$ is same as that of $S \setminus \mathcal{Z}(\phi)$.

Now let $\overline{\mathcal{Z}(\phi)}$ consists of more than one closed loop. Then by Jordan curve theorem the number of components of $\overline{S} \setminus \overline{\mathcal{Z}(\phi)}$ is at least three. This is a contradiction to Courant's nodal domain theorem \ref{courant} which says that a ${\lambda_1}(S)$-eigenfunction can have at most two nodal domains. Hence we conclude that $\overline{\mathcal{Z}(\phi)}$ is a simple closed curve in $\overline{S}$. In particular, we have the following description of $\mathcal{Z}(\phi)$ at any puncture.
 \begin{claim}\label{arcs}
If one of the punctures $p$ of $S$ is a vertex of $\overline{\mathcal{Z}(\phi)}$ then the number of arcs in $\overline{\mathcal{Z}(\phi)}$ emanating from $p$ is at most two.
\end{claim}

Let ${\lambda_1}(S)= s(1-s)$ with $s \in (\frac{1}{2}, 1]$. Let $p$ be one of the punctures of $S$. Let $\mathcal{P}^t$ be a cusp around $p$ (see \S1). Recall that $S$ being a punctured sphere, does not have any cuspidal eigenvalue \cite{Hu}, \cite{O}. Thus any ${\lambda_1}(S)$-eigenfunction $\phi$ is a linear combination of residues of Eisenstein series (see \cite{I}). It follows from \cite[Thorem 6.9]{I} that the $y^s$ term can not occur in the Fourier development (see \eqref{expres}) of these residues in $\mathcal{P}^t$. Hence $\phi$ has a Fourier development in $\mathcal{P}^t$ of the form (see \S1):
\begin{equation}\label{express}
\phi(x, y)= {\phi_0}{y^{1-s}} + {\sum_{j \geq 1}}\sqrt{\frac{2jy}{\pi}}{K_{s-\frac{1}{2}}}(jy)({\phi^e_j}\cos(j.x) + {\phi^o_j}\sin(j.x)).
\end{equation}
Now we consider the space $\mathcal{E}_1$ generated by ${\lambda_1}(S)$-eigenfunctions. The map $\pi: {\mathcal{E}_1} \to {\mathbb{R}^3}$ given by $\pi(\phi) = ({\phi_0}, {\phi^e_1}, {\phi^o_1})$ is linear and so if $\dim{\mathcal{E}_1} > 3$ then $\ker{\pi}$ is non-empty.

Let $\psi \in \ker{\pi}$ i.e. ${\psi_0} = {\psi^e_1} = {\psi^o_1} = 0$. Then by the result \cite{Ju} of Judge, the number of arcs in $\mathcal{Z}(\psi)$ emanating from $p$ is at least four, a contradiction to claim \ref{arcs}.
\section*{Acknowledgement}
I would like to thank my advisor Jean-Pierre Otal for all his help starting from suggesting the problem to me. I am thankful to Peter Buser and Werner Ballmann for the discussions that I had with them on this problem. I would like to thank the Max Planck Institute for Mathematics in Bonn
for its support and hospitality.

\appendix
\section{}
For the convenience of the reader we give a proof of the fact that, for $(g, n) \neq (0, 4), (1, 1)$, the complement ${\mathcal{M}_{g, n}} \setminus \mathcal{I}_\epsilon$ of the compact set $\mathcal{I}_\epsilon= \{S \in {\mathcal{M}_{g, n}}: s(S) \ge \epsilon \}$ \cite{B} is path connected.
\begin{lem}\label{pconn}
For any $(g, n) \neq (0, 4), (1, 1)$ with $2g-2+n > 0$ and any $\epsilon>0$ the set ${\mathcal{M}_{g, n}} \setminus \mathcal{I}_\epsilon$ is path connected.
\end{lem}
\begin{proof}
Let $S_1$ and $S_2$ be two surfaces in ${\mathcal{M}_{g, n}}$ such that $s(S_i) < \epsilon$. So we have simple closed geodesics $\gamma_1$ on $S_1$ and $\gamma_2$ on $S_2$ such that the length $l_{\gamma_i}$ of $\gamma_i$ is $< \epsilon$. Recall that it has always been our practise to treat $\mathcal{M}_{g, n}$ as a subset of all possible metrics on a fixed surface $S$ and the geodesics are understood to be parametric curves on $S$ that satisfy certain differential equations provided by the metric.

With this understanding let us first assume that $\gamma_1$ does not intersect $\gamma_2$. So we may consider a pants decomposition $P$ of $S$ containing both $\gamma_1$ and $\gamma_2$. Let the Fenchel-Nielsen coordinates of $S_i$ be given by $({l_j}(S_i), {\theta_j}(S_i))_{j=1}^{3g-3+n}$. Here $l_1$, $l_2$ are the length parameters along $\gamma_1, \gamma_2$ and $\theta_1, \theta_2$ are twist parameters along $\gamma_1, \gamma_2$. Then consider the path $\beta: [0, 1] \to \mathcal{T}_2$ given by:
\begin{equation*}
{l_1}(\beta(t))= \begin{cases}
   {l_1}(S_1)\:&\text{if $t \in [0, \frac{1}{2}]$}, \\
  2(1-t){l_1}(S_1) +(2t-1){l_1}(S_2) &\text{if $t \in [\frac{1}{2}, 1]$}
  \end{cases}
\end{equation*}
\begin{equation*}
{l_2}(\beta(t))= \begin{cases}
   (1-2t){l_2}(S_1) + 2t{l_2}(S_2)\:&\text{if $t \in [0, \frac{1}{2}]$}, \\
  {l_2}(S_2) &\text{if $t \in [\frac{1}{2}, 1]$}
    \end{cases}
\end{equation*}
${l_3}(\beta(t))=(1-t){l_3}(S_1) + t {l_3}(S_2)$ and ${\theta_j}(\beta(t))=(1-t){\theta_j}(S_1) + t {\theta_j}(S_2)$. Since $l_1(\beta(t)) < \epsilon$ for $t \in [0, \frac{1}{2}]$ and $l_2(\beta(t)) < \epsilon$ for $t \in [\frac{1}{2}, 1]$ we observe that $s(\beta(t)) < \epsilon$ for all $t$. The image of $\beta$ under the quotient map $\mathcal{T}_{g, n} \to \mathcal{M}_{g, n}$ produces the required path joining $S_1$ and $S_2$.

Now let us assume that $\gamma_1$ intersects $\gamma_2$. Let $\gamma$ be a simple closed geodesic that does not intersect $\gamma_1$ and $\gamma_2$. By our assumption i.e. $(g, n) \neq (0, 4), (1, 1)$ such a geodesic exists. Then by the procedure described above both $S_1$ and $S_2$ can be joined by a path in ${\mathcal{M}_{g, n}} \setminus \mathcal{I}_\epsilon$ to a surface on which $\gamma$ has length $< \epsilon$. This finishes the proof.
\end{proof}

\end{document}